\documentclass[11pt]{article} 
\usepackage{amsfonts,enumitem,amsmath,latexsym,amssymb,mathrsfs,amsthm,color}

\let\OLDthebibliography\thebibliography
\renewcommand\thebibliography[1]{
  \OLDthebibliography{#1}
  \setlength{\parskip}{0pt}
  \setlength{\itemsep}{0pt plus 0.3ex}
}

\def\numberlikeadb{\global\def\theequation{\thesection.\arabic{equation}}}
\numberlikeadb
\newtheorem{theorem}{Theorem}[section]

\newtheorem{proposition}[theorem]{Proposition}
\newtheorem{remark}[theorem]{Remark}

\begin{document}

\title{Stein operators for variables form the third and fourth Wiener chaoses}
\author{Robert E. Gaunt\footnote{School of Mathematics, The University of Manchester, Manchester M13 9PL, UK}  
}
\date{} 
\maketitle

\begin{abstract}Let $Z$ be a standard normal random variable and let $H_n$ denote the $n$-th Hermite polynomial.  In this note, we obtain Stein equations for the random variables $H_3(Z)$ and $H_4(Z)$, which represents a first step towards developing Stein's method for distributional limits from the third and fourth Wiener chaoses.  Perhaps surprisingly, these Stein equations are fifth and third order linear ordinary differential equations, respectively.  As a warm up, we obtain a Stein equation for the random variable $aZ^2+bZ+c$, $a,b,c\in\mathbb{R}$, which leads us to a Stein equation for the non-central chi-square distribution.  We also provide a discussion as to why obtaining Stein equations for $H_n(Z)$, $n\geq5$, is more challenging. 
\end{abstract}


\noindent{{{\bf{AMS 2010 Subject Classification:}}} Primary 60F05; 60G15; Secondary 60H07

\section{Introduction}

\subsection{Background}

In 1972, Stein \cite{stein} introduced a powerful technique for deriving distributional bounds for normal approximation.  Stein's method for normal approximation rests on the following characterisation of the normal distribution: $W\sim N(0,1)$ if and only if 
\begin{equation} \label{stein lemma}\mathbb{E}[f'(W)-Wf(W)]=0
\end{equation}
for all real-valued absolutely continuous functions $f$ such that $\mathbb{E}|f'(Z)|<\infty$ for $Z\sim N(0,1)$.  This characterisation leads to the so-called Stein equation: 
\begin{equation} \label{normal equation} f'(x)-xf(x)=h(x)-N h,
\end{equation} 
where $N h$ denotes $\mathbb{E}h(Z)$ for $Z\sim N(0,1)$, and the test function $h$ is real-valued. It is straightforward to verify that $f(x)=\mathrm{e}^{x^2/2}\int_{-\infty}^x[h(t)-N h]\mathrm{e}^{-t^2/2}\,\mathrm{d}t$ solves (\ref{normal equation}), and bounds on the solution and its derivatives in terms of the test function $h$ and its derivatives are given in \cite{chen,dgv}.  Evaluating both sides of (\ref{normal equation}) at a random variable $W$ and taking expectations gives
\begin{equation} \label{expect} \mathbb{E}[f'(W)-Wf(W)]=\mathbb{E}h(W)-N h.
\end{equation}
Thus, the problem of bounding the quantity $|\mathbb{E}h(W)-N h|$ has been reduced to bounding the left-hand side of (\ref{expect}).  For a detailed account of the method see the book \cite{chen}.

In recent years, one of the most significant applications of Stein's method for normal approximation has been to Gaussian analysis on Wiener space.  This body of research was initiated by \cite{np09}, in which Stein's method and Malliavin calculus are combined to derive a quantitative fourth moment theorem for the normal approximation of a sequence of random variables living in a fixed Wiener chaos. A detailed account of normal approximation by the Malliavin-Stein method can be found in the book \cite{np12}.

One of the advantages of Stein's method is that the above procedure can be extended to treat many other distributional approximations; examples include the Poisson \cite{chen 0}, gamma \cite{gaunt chi square, luk}, exponential \cite{chatterjee, pekoz1} and variance-gamma distributions \cite{gaunt vg}.  The Malliavin-Stein method is also applicable to other limits, such as the multivariate normal \cite{npr10}, exponential and Pearson families \cite{eden1,eden2}, centered gamma \cite{dp16, np09}, variance-gamma \cite{eichelsbacher}, linear combinations of centered chi-square random variables \cite{aaps16, aaps17, azmooden, npoly12}, as well as a large family of distributions, such as the uniform, log-normal and Pareto distributions, that satisfy a diffusive assumption \cite{kt12}.  These works have resulted in analogues of the celebrated fourth moment theorem for non-normal limits. 

However, despite these advances, there are many important distributional limits that fall outside the current state of the art of the Malliavin-Stein method.  As noted by \cite{peccati14}, an important class of limits for which little is understood are those of the type $P(Z)$, where $Z\sim N(0,1)$ and $P$ is polynomial of degree strictly greater than 2.  In particular, the case that $P$ is a Hermite polynomial is of particular interest, due to their fundamental role in Gaussian analysis and Malliavin calculus.  

Let $H_n(x)$ be the $n$-th Hermite polynomial, defined by $H_n(x)=(-1)^n\mathrm{e}^{x^2/2}\frac{\mathrm{d}^n}{\mathrm{d}x^n}(\mathrm{e}^{-x^2/2})$, $n\geq1$, and $H_0(x)=1$.  The first six Hermite polynomials are then
\begin{align*}&H_1(x)=x, \quad H_2(x)=x^2-1, \quad H_3(x)=x^3-3x, \quad H_4(x)=x^4-6x^2+3, \\
&H_5(x)=x^5-10x^3+15x, \quad H_6(x)=x^6-15x^4+45x^2-15.
\end{align*}
In this note, we consider the problem of extending Stein's method to the random variables $H_n(Z)$.  Of course, the case $n=1$ is very well understood.  The case $n=2$ corresponds to the centered chi-square random variable $Z^2-1$. This is a special case of the centered gamma distribution, for which the Malliavin-Stein method is also highly tractable; see \cite{acp14,ammp16,dp16,np09}.  As with Stein's method for normal approximation, at the heart of Stein's method for $H_2(Z)=Z^2-1$ is the Stein equation
\begin{equation}\label{chistein}2(1+x)f'(x)-xf(x)=h(x)-N_2h,
\end{equation}  
where $N_ih$ denotes the quantity $\mathbb{E}[h(H_i(Z))]$, $i\geq1$.  The Stein equation (\ref{chistein}) follows by applying a simple translation to the classic gamma Stein equation of \cite{diaconis, luk}.  Estimates for the solution of the gamma Stein equation \cite{dgv, dp16, gaunt chi square, luk} can then be used to bound the solution of (\ref{chistein}) and its derivatives.  The problem of approximating a random variable $W$ by $H_2(Z)$ is thus reduced to the tractable one of bounding the quantity $\mathbb{E}[2(1+W)f'(W)-Wf(W)]$.  

Recently, \cite{aaps17} obtained a Stein equation for random variables of the form $F_\infty=\sum_{i=1}^q \alpha_i(Z_i^2-1)$, where the $\alpha_i$ are real-valued constants and $Z_1,\ldots,Z_q$ are independent $N(0,1)$ random variables.  Their Stein equation for $F_\infty$ was $q$-th order if the $\alpha_i$ are all distinct, and to date no bound exists for the solution of the Stein equation when $q\geq3$; the case $q=2$ corresponds to the variance-gamma Stein equation, and bounds from \cite{dgv, gaunt vg} can be used.  Despite not being able to bound the solution of the Stein equation in general, their Stein equation motivated a Stein kernel for $F_\infty$ that was used in \cite{aaps16} to bound the 2-Wasserstein distance between a random variable $F$ belonging to the second Wiener chaos and the limit $F_\infty$.  
 
\subsection{Summary of results}

In this note, we consider the problem of extending Stein's method to $H_n(Z)$, $n\geq3$.  In particular, we study the problem of finding Stein equations for $H_n(Z)$, $n\geq3$.  Indeed, our main results are the following Stein equations (see Propositions \ref{prop2.3} and \ref{prop2.4}) for $H_3(Z)$ and $H_4(Z)$:
\begin{align}&486(4-x^2)f^{(5)}(x)-486xf^{(4)}(x)-27(8-x^2)f^{(3)}(x)\nonumber\\
\label{herm3}&\quad+99xf''(x)+6f'(x)-xf(x)=h(x)-N_3h,
\end{align}
and
\begin{align}&192(x+6)(3-x)f^{(3)}(x)+16(x+3)(x-12)f''(x)\nonumber\\
\label{herm4}&\quad+4(11x+6)f'(x)-xf(x)=h(x)-N_4h.
\end{align}
The first striking feature of the Stein equations (\ref{herm3}) and (\ref{herm4}) is that they are fifth and third order linear differential equations, respectively.  There is a dramatic increase in complexity from the Stein equations (\ref{normal equation}) and (\ref{chistein}) for $H_1(Z)$ and $H_2(Z)$ to the Stein equation for $H_3(Z)$.  Such higher order Stein operators were uncommon in the literature, with the exception of \cite{goldstein3}, but following the introduction of second order operators by \cite{gaunt vg, pekoz, pike} and a $n$-th order operator for the product of $n$ independent $N(0,1)$ random variables \cite{gaunt pn}, have appeared in several recent papers, and an overview is given in \cite{gms16}.  It is also perhaps surprising that the Stein equation for $H_4(Z)$ is of lower order than the Stein equation for $H_3(Z)$.  We shed some light on this feature in Remark \ref{rem3}.  

As (\ref{herm3}) and (\ref{herm4}) are differential equations of order greater than two, neither can be readily obtained by the standard generator \cite{barbour2, gotze} or density approaches \cite{ley, stein2}.  To obtain the Stein equations (\ref{herm3}) and (\ref{herm4}) we employ a recent technique that was introduced in Section 4.3 of \cite{gms16}.  The same technique could in principle be used to obtain Stein equations for $H_n(Z)$, $n\geq5$, although this would most likely lead to more involved calculations; see Section \ref{sec2.3} for further comments.  In Proposition \ref{prop2.1}, we also use the technique to obtain a Stein equation for the random variable $aZ^2+bZ+c$.  That Stein equation is new and serves as a simple illustration of the technique.  Moreover, this result leads to a Stein equation for the non-central chi square distribution (see Proposition \ref{prop2.2} and equation (\ref{noncenstein})).  This Stein equation is a second order differential equation and it would be relatively straightforward to solve the equation and bound its solution, which would yield the first two ingredients of Stein's method for non-central chi-square approximation.  

Solving and bounding the solutions of (\ref{herm3}) and (\ref{herm4}) is, however, far more challenging and we were unable to make progress here.  Despite recent advances on this aspect of Stein's method \cite{dgv}, such results are beyond the current state of the art.  However, the Stein equations (\ref{herm3}) and (\ref{herm4}) may yield insight into convergence in the third and fourth Wiener chaoses, in the same way that the Stein equation for linear combinations of chi-square random variables of \cite{aaps17} provided important motivation for approximation of random variables in the second Wiener chaos \cite{aaps16}.  As this note is devoted to the problem of finding Stein equations for $H_n(Z)$, we shall not explore the techniques that may be required for such an investigation, such as a suitable analogue of the Stein kernel (as used by \cite{aaps16}) or iterated gamma operators \cite{np10} that were used by \cite{eichelsbacher} in developing the Malliavin-Stein method for variance-gamma approximation.  Nevertheless, this work can be viewed as a first step towards the  fundamental problem of understanding Stein operators for general functionals of normal random variables, and to the development of Stein's method for limit distributions from the third and fourth Wiener chaoses.  




\section{Stein equations for $P(Z)$}

Throughout this section, we shall let $C^k(I)$ denote the space of functions that are $k$-times differentiable on the interval $I\subseteq \mathbb{R}$.  It is understood that, for a function $h$, $h^{(0)}\equiv h$.  The supremum norm of $f:I\rightarrow\mathbb{R}$ is denoted by $\|f\|=\|f\|_\infty=\sup_{x\in I}|f(x)|$.

\subsection{Stein equations for $aZ^2+bZ+c$ and the non-central chi-sqaure distribution}\label{sec2.1}

The following proposition gives a characterising equation for the random variable $aZ^2+bZ+c$, $a,b,c\in\mathbb{R}$, which leads to a corresponding Stein equation.  The proof provides a simple illustration of a recent technique of \cite{gms16}. 

\begin{proposition}\label{prop2.1}Let $X=aZ^2+bZ+c$, where $a,b,c\in\mathbb{R}$.  Let $f\in C^2((\alpha,\beta))$, where $(\alpha,\beta)\subset\mathbb{R}$ is the range of the quadratic $ax^2+bx+c$.  Suppose further that $\mathbb{E}|Xf^{(k)}(X)|<\infty$ and $\mathbb{E}|f^{(k)}(X)|<\infty$ for all $k=0,1,2$.  Then
\begin{equation}\label{prop2.1eqn}\mathbb{E}\big[(ab^2+4a^2(X-c))f''(X)+(2a^2-b^2-4a(X-c))f'(X)+(X-c-a)f(X)\big]=0.
\end{equation}
\end{proposition} 

\begin{proof} To simplify the calculations, we find a characterising equation for the random variable $W=aZ^2+bZ$ and then apply a translation to deduce a characterising equation for $aZ^2+bZ+c$.  Recall from (\ref{stein lemma}) that, for suitable $g$, $\mathbb{E}Zg(Z)=\mathbb{E}g'(Z)$. Then, since $\frac{\mathrm{d}}{\mathrm{d}x}\big(x^{k-1}f(w)\big)=(2ax^{k}+bx^{k-1})f'(w)+(k-1)x^{k-2}f(w)$ for $k\geq1$, where $w=ax^2+bx$, we have that
\begin{equation}\label{for121}\mathbb{E}Z^kf(W)=2a\mathbb{E}Z^{k}f'(W)+b\mathbb{E}Z^{k-1}f'(W)+(k-1)\mathbb{E}Z^{k-2}f(W), \quad k\geq1.
\end{equation}
Using (\ref{for121}) gives that
\begin{align}\mathbb{E}Wf(W)&=a\mathbb{E}Z^2f(W)+b\mathbb{E}Zf(W) \nonumber\\
&=2a^2\mathbb{E}Z^2f'(W)+ab\mathbb{E}Zf'(W)+a\mathbb{E}f(W)+2ab\mathbb{E}Zf'(W)+b^2\mathbb{E}f'(W) \nonumber\\
\label{form01}&=2a\mathbb{E}Wf'(W)+b^2\mathbb{E}f'(W)+a\mathbb{E}f(W)+ab\mathbb{E}Zf'(W).
\end{align}
Another application of (\ref{for121}) gives that
\begin{align}\label{form21}\mathbb{E}Wf(W)&=2a\mathbb{E}Wf'(W)+b^2\mathbb{E}f'(W)+a\mathbb{E}f(W)+2a^2b\mathbb{E}Zf''(W)+ab^2\mathbb{E}f''(W).
\end{align}
On rearranging (\ref{form21}) we obtain
\begin{equation}\label{form31}2a^2b\mathbb{E}Zf''(W)=\mathbb{E}Wf(W)-2a\mathbb{E}Wf'(W)-b^2\mathbb{E}f'(W)-a\mathbb{E}f(W)-ab^2\mathbb{E}f''(W),
\end{equation}
and from (\ref{form01}) applied to $f'$ we obtain
\begin{equation}\label{form41}2a^2b\mathbb{E}Zf''(W)=2a\mathbb{E}Wf'(W)-4a^2\mathbb{E}Wf''(W)-2ab^2\mathbb{E}f''(W)-2a^2\mathbb{E}f'(W).
\end{equation}
Combining (\ref{form31}) and (\ref{form41}) gives
\begin{equation*}\mathbb{E}\big[(ab^2+4a^2W)f''(W)+(2a^2-b^2-4aW)f'(W)+(W-a)f(W)\big]=0.
\end{equation*}
Finally, applying a simple translation yields (\ref{prop2.1eqn}), as required. 
\end{proof}

\begin{remark}The first part of the proof involved obtaining equation (\ref{form01}).  All terms on the right hand-side of (\ref{form01}) are expressed in terms of $W$, except for the final term $ab\mathbb{E}Zf'(W)$.  In the next part of the proof, we used the technique of \cite{gms16} to obtain another expression, equation (\ref{form31}), for $\mathbb{E}Zf''(W)$ in terms of expectations involving $W$.  The `awkward' $\mathbb{E}Zf''(W)$ term can then be cancelled out.  We use this basic strategy, together with more involved calculations, to obtain our Stein equations for $H_3(Z)$ and $H_4(Z)$.
\end{remark} 


The probability density function of the non-central chi-square distribution with $k>0$ degrees of freedom and non-centrality parameter $\lambda>0$ is given by
\begin{equation}\label{noncen}p(x)=\frac{1}{2}\mathrm{e}^{-(x+\lambda)/2}\bigg(\frac{x}{\lambda}\bigg)^{k/4-1/2}I_{\frac{k}{2}-1}(\sqrt{\lambda x}), \quad x>0,
\end{equation}
where $I_{\nu} (x) = \sum_{k=0}^{\infty} \frac{1}{\Gamma(\nu +k+1) k!} \left( \frac{x}{2} \right)^{\nu +2k}$ is a modified Bessel function of the first kind (see \cite{olver}).  For distributional properties see \cite{jkb95}.  If the random variable $Y$ has density (\ref{noncen}), we write $Y\sim \chi_{k}'^2(\lambda)$.  The non-central chi-square distribution $\chi_{k}'^2(\lambda)$ is equal in distribution to $\sum_{i=1}^kX_i^2$, where $X_1,\ldots,X_k$ are independent $N(\mu_i,1)$ random variables, $1\leq i\leq k$, and $\lambda=\sum_{i=1}^k\mu_i^2$.  With the aid of Proposition \ref{prop2.1} we can obtain the following proposition.

\begin{proposition}\label{prop2.2}Let $X$ have the  $\chi_{k}'^2(\lambda)$ distribution.  Let $f\in C^2((0,\infty))$ be such that $\mathbb{E}|Xf^{(j)}(X)|<\infty$, $j=0,1,2$, and $\mathbb{E}|f^{(k)}(X)|<\infty$, $k=0,1$.  Then
\begin{equation}\label{result3}\mathbb{E}\big[4Xf''(X)+(2k-4X)f'(X)+(X-k-\lambda)f(X)\big]=0.
\end{equation}
\end{proposition}

Proposition \ref{prop2.2} leads to the following Stein equation for the $\chi_{k}'^2(\lambda)$ distribution:
\begin{equation}\label{noncenstein}4xf''(x)+(2k-4x)f'(x)+(x-k-\lambda)f(x)=h(x)-\chi_{k}'^2(\lambda),
\end{equation}
where $\chi_{k}'^2(\lambda)h$ denotes $\mathbb{E}h(Y)$ for $Y\sim\chi_{k}'^2(\lambda)$.  



\begin{proof} Let $X_i\sim\chi_{(1)}'^2(\mu_i)$, $1\leq i\leq k$, be independent, and suppose that $\lambda=\sum_{i=1}^k\mu_i^2$. Then $X_i\stackrel{\mathcal{D}}{=}(Z+\mu_i)^2=Z^2+2\mu_iZ+\mu_i^2$, where $Z\sim N(0,1)$.  Setting $a=1$, $b=2\mu_i$ and $c=\mu_i^2$ in (\ref{prop2.1eqn}) gives that $\mathbb{E}[\mathcal{A}_{X_i}f(X_i)]=0$, $1\leq i\leq k$, where the Stein operator $\mathcal{A}_{X_i}$ is given by
\begin{equation*}\mathcal{A}_{X_i}f(x)=4xf''(x)+(2-4x)f'(x)+(x-1-\mu_i^2)f(x), \quad 1\leq i\leq k.
\end{equation*}
Since the random variables $X_1,\ldots,X_k$ are independent, $Y=\sum_{i=1}^kX_i$, and the Stein operator has linear coefficients, we can deduce (\ref{result3}) form the result of Section 4.1 of \cite{gms16}.
\end{proof}

\subsection{Stein equations for $H_3(Z)$ and $H_4(Z)$}

The following propositions lead to the Stein equations (\ref{herm3}) and (\ref{herm4}).  Some remarks are given after the proofs.

\begin{proposition}\label{prop2.3}Let $X=H_3(Z)=Z^3-3Z$.  Let $f\in C^5(\mathbb{R})$ and suppose $\mathbb{E}|X^2f^{(i)}(X)|<\infty$ for $i=3,5$, $\mathbb{E}|Xf^{(j)}(X)|<\infty$ for $j=0,2,4$, and $\mathbb{E}|f^{(k)}(X)|<\infty$ for $k=1,3,5$.  Denote this class of functions $\mathcal{F}_3$.  Then
\begin{align}&\mathbb{E}\big[486(4-X^2)f^{(5)}(X)-486Xf^{(4)}(X)-27(8-X^2)f^{(3)}(X)\nonumber\\
\label{herm33}&\quad+99Xf''(X)+6f'(X)-Xf(X)\big]=0.
\end{align}
\end{proposition}

\begin{proposition}\label{prop2.4}Let $X=H_4(Z)=Z^4-6Z^2+3$.  Let $f\in C^3((-6,\infty))$ and suppose that $\mathbb{E}|X^2f^{(i)}(X)|<\infty$ for $i=2,3$, $\mathbb{E}|Xf^{(j)}(X)|<\infty$ for $j=0,1,2,3$, and $\mathbb{E}|f^{(k)}(X)|<\infty$ for $k=1,2,3$.  Denote this class of functions $\mathcal{F}_4$.  Then
\begin{equation}\label{herm44} \mathbb{E}\big[192(X+6)(3-X)f^{(3)}(X)+16(X+3)(X-12)f''(X)+4(11X+6)f'(x)-Xf(X)\big]=0.
\end{equation}
\end{proposition}

\noindent{\emph{Proof of Proposition \ref{prop2.3}.}}  Let $W=aZ^3+bZ$.  We shall later set $a=1$ and $b=-3$.  Since $\frac{\mathrm{d}}{\mathrm{d}x}\big(x^{k-1}f(w)\big)=(3ax^{k+1}+bx^{k-1})f'(w)+(k-1)x^{k-2}f(w)$ for $k\geq1$, where $w=ax^3+bx$, we have from (\ref{stein lemma}) that
\begin{equation}\label{for1}\mathbb{E}Z^kf(W)=3a\mathbb{E}Z^{k+1}f'(W)+b\mathbb{E}Z^{k-1}f'(W)+(k-1)\mathbb{E}Z^{k-2}f(W), \quad k\geq1.
\end{equation}
Using (\ref{for1}) gives that
\begin{align*}\mathbb{E}Wf(W)&=a\mathbb{E}Z^3f(W)+b\mathbb{E}Zf(W) \\
&=3a^2\mathbb{E}Z^4f'(W)+ab\mathbb{E}Z^2f'(W)+2a\mathbb{E}Zf(W)+3ab\mathbb{E}Z^2f'(W)+b^2\mathbb{E}f'(W) \\
&=3a^2\mathbb{E}Z^4f'(W)+4ab\mathbb{E}Z^2f'(W)+2a\mathbb{E}Zf(W)+b^2\mathbb{E}f'(W).
\end{align*} 
Another application of (\ref{for1}) gives that
\begin{align*}\mathbb{E}Wf(W)&=9a^3\mathbb{E}Z^5f''(W)+3a^2b\mathbb{E}Z^3f''(W)+9a^2\mathbb{E}Z^2f'(W)+12a^2b\mathbb{E}Z^3f''(W)\\
&\quad+4ab^2\mathbb{E}Zf''(W)+4ab\mathbb{E}f'(W)+6a^2\mathbb{E}Z^2f'(W)+2ab\mathbb{E}f'(W)+b^2\mathbb{E}f'(W) \\
&=9a^3\mathbb{E}Z^5f''(W)+15a^2b\mathbb{E}Z^3f''(W)+4ab^2\mathbb{E}Zf''(W)+15a^2\mathbb{E}Z^2f'(W)\\
&\quad+(6ab+b^2)\mathbb{E}f'(W).
\end{align*}
Following the same procedure again yields
\begin{align}\mathbb{E}Wf(W)&=27a^4\mathbb{E}Z^6f^{(3)}(W)+9a^3b\mathbb{E}Z^4f^{(3)}(W)+36a^3\mathbb{E}Z^3f''(W)\nonumber\\
&\quad+45a^3b\mathbb{E}Z^4f^{(3)}(W)+15a^2b^2\mathbb{E}Z^2f^{(3)}(W)+30a^2b\mathbb{E}Zf''(W)\nonumber\\
&\quad+12a^2b^2\mathbb{E}Z^2f^{(3)}(W)+4ab^3\mathbb{E}f^{(3)}(W)\nonumber \\
&\quad+45a^3\mathbb{E}Z^3f''(W)+15a^2b\mathbb{E}Zf''(W)+15a^2\mathbb{E}f'(W)+(6ab+b^2)\mathbb{E}f'(W)\nonumber \\
&=27a^4\mathbb{E}Z^6f^{(3)}(W)+54a^3b\mathbb{E}Z^4f^{(3)}(W)+27a^2b^2\mathbb{E}Z^2f^{(3)}(W)+81a^3\mathbb{E}Z^3f''(W)\nonumber\\
&\quad+45a^2b\mathbb{E}Zf''(W)+4ab^3\mathbb{E}f^{(3)}(W)+(15a^2+6ab+b^2)\mathbb{E}f'(W)\nonumber \\
&=27a^2\mathbb{E}W^2f^{(3)}(W)+4ab^3\mathbb{E}f^{(3)}(W)+(15a^2+6ab+b^2)\mathbb{E}f'(W)\nonumber\\
&\quad+81a^3\mathbb{E}Z^3f''(W)+45a^2b\mathbb{E}Zf''(W)\nonumber \\
&=27a^2\mathbb{E}W^2f^{(3)}(W)+4ab^3\mathbb{E}f^{(3)}(W)+(15a^2+6ab+b^2)\mathbb{E}f'(W)\nonumber\\
\label{ex1}&\quad+81a^2\mathbb{E}Wf''(W)-36a^2b\mathbb{E}Zf''(W).
\end{align}
We now note that
\begin{align}\mathbb{E}Zf''(W)&=3a\mathbb{E}Z^2f^{(3)}(W)+b\mathbb{E}f^{(3)}(W)\nonumber \\
&=9a^2\mathbb{E}Z^3f^{(4)}(W)+3ab\mathbb{E}Zf^{(4)}(W)+(3a+b)\mathbb{E}f^{(3)}(W)\nonumber \\
\label{eqn1}&=9a\mathbb{E}Wf^{(4)}(W)-6ab\mathbb{E}Zf^{(4)}(W)+(3a+b)\mathbb{E}f^{(3)}(W).
\end{align}
On substituting (\ref{eqn1}) into (\ref{ex1}) we obtain
\begin{align}\mathbb{E}Wf(W)&=27a^2\mathbb{E}W^2f^{(3)}(W)+4ab^3\mathbb{E}f^{(3)}(W)+(15a^2+6ab+b^2)\mathbb{E}f'(W)\nonumber\\
&\quad+81a^2\mathbb{E}Wf''(W)-324a^3b\mathbb{E}Wf^{(4)}(W)-(108a^3b+36a^2b^2)\mathbb{E}f^{(3)}(W)\nonumber\\
\label{ex2}&\quad+216a^3b^2\mathbb{E}Zf^{(4)}(W).
\end{align}
On rearranging (\ref{ex2}) we obtain
\begin{align}216a^3b^2\mathbb{E}Zf^{(4)}(W)&=\mathbb{E}Wf(W)+324a^3b\mathbb{E}Wf^{(4)}(W)\nonumber\\
&\quad+(108a^3b+36a^2b^2-4ab^3)\mathbb{E}f^{(3)}(W)-27a^2\mathbb{E}W^2f^{(3)}(W)\nonumber\\
\label{ex3}&\quad-81a^2\mathbb{E}Wf''(W)-(15a^2+6ab+b^2)\mathbb{E}f'(W),
\end{align}
and from (\ref{ex1}) applied to $f''$ we obtain
\begin{align}216a^3b^2\mathbb{E}Zf^{(4)}(W)&=6ab\big\{27a^2\mathbb{E}W^2f^{(5)}(W)+4ab^3\mathbb{E}f^{(5)}(W)+81a^2\mathbb{E}Wf^{(4)}(W)\nonumber\\
&\quad+(15a^2+6ab+b^2)\mathbb{E}f^{(3)}(W)-\mathbb{E}Wf''(W)\big\}\nonumber\\ 
&=162a^3b\mathbb{E}W^2f^{(5)}(W)+24a^2b^4\mathbb{E}f^{(5)}(W)+486a^3b\mathbb{E}Wf^{(4)}(W)\nonumber\\
\label{ex4}&\quad+(90a^3b+36a^2b^2+6ab^3)\mathbb{E}f^{(3)}(W)-6ab\mathbb{E}Wf''(W).
\end{align}
Combining (\ref{ex3}) and (\ref{ex4}) now yields
\begin{align*}&\mathbb{E}\big[(162a^3bW^2+24a^2b^4)f^{(5)}(W)+162a^3bWf^{(4)}(W)+(27a^2W^2-18a^3b+2ab^3)f^{(3)}(W)\\
&\quad+(81a^2-6ab)Wf''(W)+(15a^2+6ab+b^2)f'(W)-Wf(W)\big]=0.
\end{align*}
Finally, on setting $a=1$ and $b=-3$ we obtain (\ref{herm33}), as required. \hfill $\Box$

\vspace{3mm}


\noindent{\emph{Proof of Proposition \ref{prop2.4}.}} To simplify the calculations, we find a characterising equation for the random variable $W=aZ^4+bZ^2$ and then apply a translation to deduce a characterising equation for the random variable $aZ^4+bZ^2+c$.  We shall then set $a=1$, $b=-6$ and $c=3$.  Since $\frac{\mathrm{d}}{\mathrm{d}x}\big(x^{k-1}f(w)\big)=(4ax^{k+2}+2bx^{k})f'(w)+(k-1)x^{k-2}f(w)$ for $k\geq1$, where $w=ax^4+bx^2$, we have from (\ref{stein lemma}) that
\begin{equation}\label{for12}\mathbb{E}Z^kf(W)=4a\mathbb{E}Z^{k+2}f'(W)+2b\mathbb{E}Z^{k}f'(W)+(k-1)\mathbb{E}Z^{k-2}f(W), \quad k\geq1.
\end{equation}
Using (\ref{for12}) gives that
\begin{align*}\mathbb{E}Wf(W)&=a\mathbb{E}Z^4f(W)+b\mathbb{E}Z^2f(W) \\
&=4a^2\mathbb{E}Z^6f'(W)+2ab\mathbb{E}Z^4f'(W)+3a\mathbb{E}Z^2f(W)\\
&\quad+4ab\mathbb{E}Z^4+2b^2\mathbb{E}Z^2f'(W)+b\mathbb{E}f(W) \\
&=4a^2\mathbb{E}Z^6f'(W)+6ab\mathbb{E}Z^4f'(W)+3a\mathbb{E}Z^2f(W)+2b^2\mathbb{E}Z^2f'(W)+b\mathbb{E}f(W).
\end{align*} 
Another application of (\ref{for12}) gives that
\begin{align}\mathbb{E}Wf(W)&=16a^3\mathbb{E}Z^8f''(W)+8a^2b\mathbb{E}Z^6f''(W)+20a^2\mathbb{E}Z^4f'(W)
+24a^2b\mathbb{E}Z^6f''(W)\nonumber\\
&\quad+12ab^2\mathbb{E}Z^4f''(W)+18ab\mathbb{E}Z^2f'(W)+8ab^2\mathbb{E}Z^4f''(W)+4b^3\mathbb{E}Z^2f''(W)
\nonumber\\
&\quad+2b^2\mathbb{E}f'(W)+12a^2\mathbb{E}Z^4f'(W)+6ab\mathbb{E}Z^2f'(W)+3a\mathbb{E}f(W)+b\mathbb{E}f(W)\nonumber \\
&=16a^3\mathbb{E}Z^8f''(W)+32a^2b\mathbb{E}Z^6f''(W)+20ab^2\mathbb{E}Z^4f''(W)+4b^3\mathbb{E}Z^2f''(W)\nonumber\\
&\quad
+32a^2\mathbb{E}Z^4f'(W)+24ab\mathbb{E}Z^2f'(W)+2b^2\mathbb{E}f'(W)+(3a+b)\mathbb{E}f(W) \nonumber\\
&=16a\mathbb{E}W^2f''(W)+4b^2\mathbb{E}Wf''(W)+32a\mathbb{E}Wf'(W)+2b^2\mathbb{E}f'(W)\nonumber\\
\label{form0}&\quad+(3a+b)\mathbb{E}f(W)-8ab\mathbb{E}Z^2f'(W).
\end{align}
We now note that
\begin{align}\mathbb{E}Z^2f'(W)&=4a\mathbb{E}Z^4f''(W)+2b\mathbb{E}Z^2f''(W)+\mathbb{E}f'(W)\nonumber\\
\label{form1}&=4\mathbb{E}Wf''(W)-2b\mathbb{E}Z^2f''(W)+\mathbb{E}f'(W).
\end{align}
On substituting (\ref{form1}) into (\ref{form0}) we obtain
\begin{align}\mathbb{E}Wf(W)&=16a\mathbb{E}W^2f''(W)+(4b^2-32ab)\mathbb{E}Wf''(W)+32a\mathbb{E}Wf'(W)\nonumber\\
\label{form2}&\quad+(2b^2-8ab)\mathbb{E}f'(W)+(3a+b)\mathbb{E}f(W)+16ab^2Z^2f''(W).
\end{align}
On rearranging (\ref{form2}) we obtain
\begin{align}16ab^2\mathbb{E}Z^2f''(W)&=\mathbb{E}Wf(W)-16a\mathbb{E}W^2f''(W)+(32ab-4b^2)\mathbb{E}Wf''(W)\nonumber\\
\label{form3}&\quad-32a\mathbb{E}Wf'(W)+(8ab-2b^2)\mathbb{E}f'W)-(3a+b)\mathbb{E}f(W),
\end{align}
and from (\ref{form0}) applied to $f'$ we obtain
\begin{align}16ab^2\mathbb{E}Z^2f''(W)&=2b\big\{-\mathbb{E}Wf'(W)+16a\mathbb{E}W^2f^{(3)}(W)+4b^2\mathbb{E}Wf^{(3)}(W)\nonumber\\
&\quad+32a\mathbb{E}Wf''(W)+2b^2\mathbb{E}f''(W)+(3a+b)\mathbb{E}f'(W)\big\} \nonumber\\
&=-2b\mathbb{E}Wf'(W)+32ab\mathbb{E}W^2f^{(3)}(W)+8b^3\mathbb{E}Wf^{(3)}(W)\nonumber\\
\label{form4}&\quad+64ab\mathbb{E}Wf''(W)+4b^3\mathbb{E}f''(W)+(6ab+2b^2)\mathbb{E}f'(W).
\end{align}
Combining (\ref{form3}) and (\ref{form4}) now yields
\begin{align*}&\mathbb{E}\big[(32abW^2+8b^3W)f^{(3)}(W)+(16aW^2+(4b^2+32ab)W+4b^3)f''(W)\\
&\quad+((32a-2b)W+4b^2-2ab)f'(W)+(3a+b-W)f(W)\big]=0.
\end{align*}
By applying a translation, we deduce the following characterising equation for the random variable $X=aZ^4+bZ^2+c$:
\begin{align*}&\mathbb{E}\big[(32ab(X-c)^2+8b^3(X-c))f^{(3)}(X)+(16a(X-c)^2+(4b^2+32ab)(X-c)+4b^3)f''(X)\\
&\quad+((32a-2b)(X-c)+4b^2-2ab)f'(X)-(X-3a-b-c)f(X)\big]=0.
\end{align*}
Finally, on setting $a=1$, $b=-6$ and $c=3$ we obtain (\ref{herm44}), as required. \hfill $\Box$


\begin{remark}\label{remmmm}In the proofs of Propositions \ref{prop2.3} and \ref{prop2.4}, we actually obtained characterising equations for the more general random variables $aZ^3+bZ$ and $aZ^4+bZ^2+c$.  The reason for working with these more general random variables is that it enables one to combine terms of the type $\mathbb{E}Z^if^{(j)}(W)$ to form terms of the type $\mathbb{E}W^kf^{(j)}(W)$.  For example, \emph{a priori} it is not immediately clear whether the term $\mathbb{E}Z^3f^{(3)}(W)$ should be combined to form $\mathbb{E}W^3f^{(3)}(W)$ or $\mathbb{E}Wf^{(3)}(W)$, but when multiplied by an appropriate constant in terms of $a$ and $b$ this become clear.  Indeed, a reasonable amount of trial and error was used in deriving the characterising equations for $aZ^3+bZ$ and $aZ^4+bZ^2+c$, and the introduction of the constants $a$, $b$ and $c$ greatly helped in this regard. 

It is natural to ask for Stein operators for the more general polynomials $P_3(Z)=aZ^3+bZ^2+cZ+d$ and $P_4(Z)=aZ^4+bZ^3+cZ^2+dZ+e$.  However, this leads to non-trivial further complications.  A discussion on why there is an increase in difficulty for the general cubic $P_3(Z)$ is given in Section \ref{sec2.3}.  That there is an increase in difficulty for the general quartic $P_4(Z)$ is easy to see: a Stein equation for $P_4(Z)$ must include the fifth order Stein equation for $aZ^3+bZ$ as a special case.
\end{remark}

\begin{remark}It should be possible, via the density method \cite{stein2, ley} for example, to obtain Stein equations for $H_3(Z)$ and $H_4(Z)$ that are differential equations with order less than five and three respectively, although this would most likely come at the cost of more complicated coefficients than the polynomial coefficients of (\ref{herm3}) and (\ref{herm4}).   We strongly believe that (\ref{herm3}) and (\ref{herm4}) are the minimal order Stein equations with polynomial coefficients for $H_3(Z)$ and $H_4(Z)$, in the sense that there do not exist Stein equations for $H_3(Z)$ and $H_4(Z)$ that are polynomial coefficient differential equations of order less than five and three, respectively.  We are unable to prove this, but a future work will investigate this assertion for polynomials of a fixed low degree.



\end{remark}

\begin{remark} Let us now provide some insight regarding the orders of the Stein operators obtained in this note.  It is perhaps surprising that our Stein operator for the non-central chi-square distribution is of second order, whereas the classical Stein operator for the centered chi-square distribution is first order order.  This is a consequence of the duality between Stein operators and differential equations satisfied by the p.d.f$.$ of the distribution of interest.  Roughly speaking, if the p.d.f$.$ satisfies a $k$-th order differential equation with polynomial coefficients, then, under fairly mild assumptions, there will exist a $k$-th order Stein operator for the distribution (see \cite{gaunt ngb, gms16,ley} for more details).  The presence of the modified Bessel function of the first kind (which is known to satisfy a second order differential equation with polynomial coefficients) means that the p.d.f$.$ of the non-central chi-square distribution satisfies a second order differential equation, and so by duality we arrive at a second order Stein operator.  In contrast, the p.d.f$.$ of the centered chi-square distribution is given in terms of elementary functions and satisfies a first order differential equation. 

We now provide two pieces of intuition as to why our Stein operator for $H_3(Z)$ is of greater order than our Stein operator for $H_4(Z)$.  One is again due to the duality between p.d.f.s and Stein operators.  The p.d.f.s for $H_3(Z)$ and $H_4(Z)$ can be obtained via elementary, but tedious, calculations.  We omit the details, but make the observation that the p.d.f.s are defined piecewise, which helps explain why the Stein operators are of order greater than one (see pp$.$ 574--575 of \cite{pike} for details regarding the derivation of Stein equations for the piecewisely defined Laplace distribution).  To see this, note that in order to find the p.d.f.s of $H_3(Z)$ and $H_4(Z)$ via the usual method for finding p.d.f.s of functions of random variables one must solve the equations $z^3-3z=x$ and $z^4-6z^2+3=x$, respectively, for $z$, and the nature of the solutions will change as $x$ varies over the real line.  The solution of $z^3-3z=x$ is more complicated than that of $z^4-6z^2+3=x$, as it requires the use of the cubic formula rather than the quadratic formula, and consequently the p.d.f$.$ of $H_3(Z)$ will take a more complicated form than that of $H_4(Z)$.  It therefore seems plausible that the p.d.f$.$ of $H_3(Z)$ will only satisfy, in a piecewise sense, a polynomial coefficient differential equation that is of higher order than that of $H_4(Z)$, meaning by duality that the Stein operator will also be of higher order.

Another piece of intuition is that obtaining a Stein equation for $H_4(Z)$ is no more difficult than obtaining one for $Z^4-6Z^2$.  Now $Z^2$ has the $\chi_{(1)}^2$ distribution, for which the classical Stein equation is first order.  Therefore finding a Stein equation for $Z^4-6Z^2$ is equivalent to finding one for $Y^2-6Y$, where $Y\sim\chi_{(1)}^2$, which one would not expect to be much more involved than finding one for $aZ^2+bZ$, which was done in Proposition \ref{prop2.1}.
\end{remark}


\begin{remark}\label{rem3}We were unable to solve the Stein equations (\ref{herm3}) and (\ref{herm4}), and as a result have not been able to bound the solutions in terms of supremum norms of the test function $h$.  Therefore we cannot use the standard method for proving sufficiency of Stein characterisations (see, for example, the proof of Lemma 2.1 of \cite{chen}).  Moreover, the law of $H_n(Z)$, $n\geq3$, is not determined by its moments.  Therefore, we cannot appeal to Proposition 4.8 of \cite{gms16} to prove sufficiency.  (Note that one can perfecly well apply Stein's method without establishing sufficiency; see, for example, \cite{pekoz}.)  However, we can use the argument used to prove Proposition 4.8 of \cite{gms16} to obtain a weak form of sufficiency. If for a random variable $W_3$, which shares its first five moments with $H_3(Z)$, the expectation (\ref{herm33}) holds for all $f\in\mathcal{F}_3$, as defined in Proposition \ref{prop2.3}, then $W_3$ must share all moments with $H_3(Z)$, and if for a random variable $W_4$, which shares its first three moments with $H_4(Z)$, the expectation (\ref{herm44}) holds for all $f\in\mathcal{F}_4$ then $W_4$ must share all moments with $H_4(Z)$. 
\end{remark}

\subsection{Stein equations for $H_n(Z)$, $n\geq5$}\label{sec2.3}

The next obvious problems are to find Stein equations for $P_3(Z)$, where $P_3$ is a general cubic, and $H_n(Z)$, $n\geq5$.  Based on our Stein equations for $aZ^2+bZ+c$, $H_3(Z)$ and $H_4(Z)$, we expect that the order of increasing difficulty is: $P_3(Z)$, $H_6(Z)$ and then $H_5(Z)$.  Before tackling the general case of $P_3(Z)$, it may be worth working out the special case $W=aZ^3+bZ^2$.  Even here one faces additional difficulties.  Unlike in our proofs of Propositions \ref{prop2.1}, \ref{prop2.3} and \ref{prop2.4}, it is not possible for any $i\geq1$ to find a $j\geq1$ such that $\mathbb{E}Z^if(W)=\mathbb{E}Z^if^{(j)}(W)+R$, where $R$ is an expression that can be solely written in terms of expectations involving $W$.  (This author has encountered this same problem whilst trying to obtain Stein operators for various product distributions, and considers it to be important for methods to be developed to overcome this difficulty.)  This means that a more sophisticated implementation of the technique of \cite{gms16} may be needed to deal with the `awkward' expressions $\mathbb{E}Z^if^{(k)}(W)$.  Even with a more ingenious implementation of this technique, the complexity of the calculations would most likely mean that the approach would become intractable for $H_7(Z)$, $H_8(Z)$ and beyond.  Obtaining Stein equations for such random variables will therefore most likely require the introduction of new techniques.

The following observation may be useful in the quest for Stein equations for $H_n(Z)$, $n\geq5$. In Table 1, we record the coefficient of the highest order derivative of the Stein equation for $H_n(Z)$ for $n=1,\ldots,4$, and note its connection to the local maxima and minima of the Hermite polynomials.  This allows us to conjecture the corresponding terms for the Stein equations for $H_5(Z)$ and $H_6(Z)$. For $H_5(Z)$ we expect the leading coefficient to be
\begin{align*}&\Big(x-4\sqrt{6(3+\sqrt{6})}\Big)\Big(x-4\sqrt{6(3+\sqrt{6})}\Big)\Big(x+4\sqrt{6(3+\sqrt{6})}\Big)\Big(x+4\sqrt{6(3-\sqrt{6})}\Big)\\
&=x^4-576x^2+27648,
\end{align*}
and for $H_6(Z)$ we expect it to be
\begin{equation*}\big(x-20(\sqrt{10}-2)\big)\big(x+20(2+\sqrt{10})\big)\big(x+15\big)=x^3+95x^2-1200x-3600.
\end{equation*}

In the quest for Stein equations for $H_n(Z)$, $n\geq5$, it would also be useful to have a conjecture on the order of the Stein equations.  This author does not have a conjecture on the exact order, but does strongly suspect that the Stein equations for $H_5(Z)$ and $H_6(Z)$ will be of at least ninth and sixth order, respectively.   

\begin{table}[ht]
\caption{Local maxima and minima of Hermite polynomials $H_n(x)$ and coefficient of highest order derivative of the Stein equation for $H_n(Z)$.} 
\centering

\begin{tabular}{ |p{1cm}|p{4.8cm}|p{5.4cm}|p{2.4cm}|  }
 \hline
  & Local maxima & Local minima & Coefficient\\
  \hline
 $H_1(x)$   & -    & - & 1 \vspace{1mm} \\
 $H_2(x)$ & - & $-1$ & $x$ \vspace{1mm} \\
 $H_3(x)$   & $2$ & $-2$ &  $x^2-4$ \vspace{1mm} \\
 $H_4(x)$  & 3 & $-6$ (twice) & $(x+6)(x-3)$ \\
 $H_5(x)$  & $4\sqrt{6(3+\sqrt{6})}$, $4\sqrt{6(3-\sqrt{6})}$ & $-4\sqrt{6(3+\sqrt{6})}$, $-4\sqrt{6(3-\sqrt{6})}$ & ? \vspace{1mm}\\
 $H_6(x)$  & $20(\sqrt{10}-2)$ (twice) & $-20(2+\sqrt{10})$ (twice), $-15$ & ? \vspace{1mm}\\
 \hline
\end{tabular}
\label{dag}
\end{table}

\subsection*{Acknowledgements}
The author is supported by a Dame Kathleen Ollerenshaw Research Fellowship.  The author would like to thank Ivan Nourdin for suggesting the problem of finding Stein equations for $H_n(Z)$, $n\geq 3$.  The author would also to thank the reviewer for their comments and suggestions which lead to an improved article, and is grateful to Ehsan Azmoodeh and Dario Gasbarra for bringing to my attention a mistake in the Stein equation for $H_4(Z)$ that was present in an earlier version.

\end{document}